\documentclass[11pt]{amsart}
\usepackage{mathrsfs}
\usepackage{amssymb}
\usepackage{array}
\usepackage{color}
\usepackage{amsthm}
\usepackage{amsmath}
\usepackage{graphicx}
\usepackage{cases}
\usepackage[pagewise]{lineno}
\usepackage[square, comma, sort&compress, numbers]{natbib}
\numberwithin{equation}{section}

\newtheorem{thm}{Theorem}[section]
\newtheorem{cor}[thm]{Corollary}
\newtheorem{lem}[thm]{Lemma}
\newtheorem{prop}[thm]{Proposition}
\theoremstyle{definition}
\newtheorem{defn}{Definition}[section]
\newtheorem{rem}{Remark}[section]

\usepackage{hyperref}

\begin{document}
	\title[Second boundary value Lagrangian mean curvature type equation]{On the second boundary value problem for Lagrangian mean curvature type equation }

	 \author[J.-G.~Bao]{Jiguang Bao}
	 \thanks{The first author is supported by the National Natural Science Foundation of China (12371200) and Beijing Natural
	 	Science Foundation (1254049).}
	\author[Q.-F.~Jiang]{Qinfeng Jiang*}
	\thanks{*Corresponding author.}

	\maketitle
	
	\begin{center}
		\normalsize
		School of Mathematical Sciences, Beijing Normal University,\\
		Beijing, 100875, China\\[0.3em]
		\texttt{jgbao@bnu.edu.cn}\\
		\texttt{202531130031@mail.bnu.edu.cn}
	\end{center}

		\begin{abstract}
		This article is concerned with the second boundary value problem of the Lagrangian mean curvature type
		equation arising from special Lagrangian geometry. By the parabolic method, we consider a fully nonlinear parabolic equation with oblique  derivative boundary condition, and show the long time existence and convergence of the flow. It follows that the existence and uniqueness of the smooth uniformly convex solution are obtained, which generalizes the Brendle--Warren's theorem about minimal Lagrangian diffeomorphism in Euclidean metric space.		
	\end{abstract}

\section{Introduction}
\label{Section1}
The main aim of this article is to study the existence and uniqueness of the smooth uniformly
convex solution for the second boundary value problem of the Lagrangian mean curvature type
equation
\begin{numcases}{} 
	\sum_{i=1}^{n}\arctan \left(\lambda_{i}(D^2u)\right)=\iota \cdot Du+\kappa\cdot x+c, \quad x\in\Omega,\label{EQE}\\ 
	Du(\Omega)=\tilde{\Omega},\label{DuO}
\end{numcases}
where $\iota, \kappa$ are two known vectors, $c$ is a constant to be determined, $Du$ and $D^2u$ are the gradient and the Hessian matrix of the function $u$, respectively, $\Omega$ and $\tilde{\Omega}$ are two uniformly convex bounded domains with smooth boundary in $\mathbb{R}^n$ and $\lambda(D^2u)=(\lambda_1,\cdots,\lambda_n)$ are the eigenvalues of $D^2u.$ 

\eqref{EQE} is a variant of Lagrangian mean curvature equation. Under the framework of calibrated geometry, the special Lagrangian equation 
\begin{equation}\label{Fc}
	\sum_{i=1}^{n}\arctan \left(\lambda_{i}(D^2u)\right)=c,
\end{equation}
was first introduced by Harvey--Lawson in \cite{Harvey1982CalibratedG} back in 1982. Its solutions $u$ were shown to have the property that the gradient graph $(x, Du(x))$
in Euclidean space is a Lagrangian submanifold which is absolutely volume--minimizing, and the linearization at any solution is elliptic. 
Several methods for studying the Bernstein type
theorems occured in the literature \cite{MR1938816,MR1930884}, the result is the entire smooth convex solutions of \eqref{Fc} must be a quadratic polynomial. Jost--Xin \cite{MR1938816} obtained it by using harmonic maps into convex subsets of Grassmannians and Yuan's method \cite{MR1930884}  
based on the geometric measure theory.  Chen--Warren--Yuan \cite{Chen2009} established interior regularity,  and showd convex viscosity solutions of \eqref{Fc} are smooth in the interior of the domain when $|c|\geq \frac{(n-1)\pi}{2}$. Moreover, all convex viscosity solutions to \eqref{Fc} are real analytic \cite{Chen2023}.

Over the past few years, significant progress has been made in proving the existence of minimal Lagrangian graphs (where \(\iota, \kappa \equiv 0\) in \eqref{EQE}--\eqref{DuO}), with \(Du\) being a diffeomorphism from \(\Omega\) to \(\tilde{\Omega}\). That is,
\begin{equation}\label{Ftauc}
	\left.\left\{\begin{array}{l}
		\sum\limits_{i=1}^{n}\arctan \left(\lambda_{i}(D^2u)\right)=c,\quad x\in\Omega,\\Du(\Omega)=\tilde{\Omega}.
	\end{array}\right.
	\right.
\end{equation}
Brendle--Warren \cite{Brendle2008ABV} proved the existence and uniqueness of the
solution by the elliptic method, and Huang \cite{Huang2014OnTS} obtained the existence
of solution by considering the second boundary value problem for Lagrangian mean
curvature flow.  Under a certain bound on the Lipschitz norm of an initial
entire Lagrangian graph, Chau--Chen--He \cite{Chau2012} obtained the  Lagrangian mean curvature flow of entire Lagrangian graphs has a smooth longtime solution.

To deal with the second boundary value problem \eqref{EQE}-\eqref{DuO}, the details without the item $Du$ in \eqref{EQE} can be seen
in \cite{Wang2023} and the corresponding parabolic version results in \cite{Wang2024}. That is, for the second boundary value problem of the
Lagrangian mean curvature equation
\begin{equation}\label{preeq}
	\begin{cases} 
		\sum\limits_{i=1}^{n}\arctan \left(\lambda_{i}(D^2u)\right)=\kappa\cdot x+c,\quad  x\in\Omega,
		\\ Du(\Omega)=\tilde{\Omega},
	\end{cases}
\end{equation}
its Lagrangian graph $(x, Du(x))$ prescribed
constant mean curvature vector $H = (0,\kappa)^{\perp}$ in Euclidean space such that $Du$ is the diffeomorphism
between two uniformly convex bounded domains.

Bhattacharya--Shankar \cite{bhattacharya2024optimalregularitylagrangianmean} found if $u$ solves \eqref{EQE}, then $X(x,t) = (x,Du(x))+t(-\iota,\kappa)$ is a translator solution of $(X_t)^\perp=\Delta_gX$, where 
\[\Delta_{g}=\sum_{i,j=1}^{n}\frac{1}{\sqrt{\det g}}\partial_{i}\left(\sqrt{\det g}g^{ij}\partial_{j}\right) \]
is the Laplace--Beltrami operator of the induced metric $g=(g_{ij})=\left(I+D^{2}uD^{2}u\right)$ with $(g^{ij})=(g_{ij})^{-1}$, $\perp$ is the normal projection and  $X(x,t):\mathbb{R}^n\times\mathbb{R}\to\mathbb{C}^n$ are a family of Lagrangian submanifolds evolves by Lagrangian mean curvature flow with constant mean curvature $H = (-\iota,\kappa)^{\perp}$.  They also obtained the regularity for convex viscosity solutions to \eqref{EQE} in \cite{Bhattacharya2023,bhattacharya2024optimalregularitylagrangianmean}.

Inspired by \cite{Wang2023}, we get the existence and uniqueness of the smooth uniformly convex solution to \eqref{EQE}-\eqref{DuO} as follows
\begin{thm}\label{prethm}
	There exists some positive constant $\epsilon_{0}$ depending only
	on $\Omega$ and $\tilde{\Omega}$, such that if $|\iota|\leq \epsilon_{0},|\kappa|\leq \epsilon_{0}$, then there exist a uniformly convex solution
	$u\in C^\infty(\overline{\Omega})$ and a unique constant $c$ solving \eqref{EQE}-\eqref{DuO}, and $u$ is unique up to a constant.
\end{thm}

Theorem \ref{prethm} exhibits an extension of the previous work on $\iota=0, \kappa = 0$ done by Brendle--
Warren \cite{Brendle2008ABV}, Huang \cite{Huang2014OnTS}, Huang--Ou \cite{MR3707986}, Huang--Ye \cite{MR4012119} and Chen--Huang--Ye \cite{MR3853140}, respectively. Recengtly, the work on $\iota=0$ done by \cite{Wang2023}.

Since Thomas--Yau \cite{Thomas20021075} developed the mean curvature flow of Langrangian submanifolds of Calabi--Yau manifolds, Lagrangian mean curvature flow has been studied by many authors. Smoczyk--Wang obtained the long time existence and convergence of Lagrangian mean curvature flow in some conditions (cf.\cite{Smoczyk200425,Smoczyk2002243}). There are some works according to solving elliptic equations with second boundary conditions by parabolic approach. Kitagawa \cite{Kitagawa+2012+127+160} considered a parabolic version of the mass transport problem, and showed that a solution converges to a solution of the optimal transport problem as $t$ tends to positive infinity, which have the global $C^{2,\alpha}$ and $W^{2,p}$ regularity  worked by Chen--Liu--Wang \cite{Chensb2021}. In \cite{Schnrer2003NeumannAS}, Schnurer--Smoczyk showed that the flow exists for all time and converges eventually to the solution of the prescribed Gauss curvature equation.

Motivated by the above works and translating solution results of Altschuler--Wu \cite{Altschuler1994} and Schnürer \cite{Schnürer2002},  we consider the following more general Lagrangian mean curvature flow  
\begin{equation}\label{EQP}
	\frac{\partial u}{\partial t}=\sum_{i=1}^{n}\arctan \left(\lambda_{i}(D^2u)\right)-f\left(x,Du\right),\quad x\in\Omega,\quad t>0,
\end{equation}
associated with the second boundary value condition
\begin{equation}\label{BC}
	Du(\Omega)=\tilde{\Omega},\quad x\in\partial\Omega,\quad t>0,
\end{equation}
and the initial condition
\begin{equation}\label{IC}
	u=u_0,\quad x\in\Omega, \quad t=0,
\end{equation}
where $f\in C^{2+\alpha}(\overline{\Omega}\times\overline{\tilde{\Omega}}), 0<\alpha<1$ and $u_0\in C^{2+\alpha}(\overline{\Omega})$ is a uniformly convex function satisfying $Du_0(\Omega)=\tilde{\Omega}$. 
Denote
$$\mathscr{A}:=\left\{f(x,p)\in C^{2+\alpha}(\overline{\Omega}\times\overline{\tilde{\Omega}}): f ~is~ concave ~in~x~ ~and ~convex~in~p\right\}.$$


Next, we give the long time existence and convergence results of the flow \eqref{EQP}--\eqref{IC}.
\begin{thm}\label{main-thm}
	Let $f \in \mathscr{A}$, if $|D_xf|,|D_pf|$ and $|D_{xp}f|$ are sufficiently small, then the
	uniformly convex solution of \eqref{EQP}--\eqref{IC} exists for all $ t \geq 0$ and $u(\cdot,t)$ converges to
	a function $u^{\infty}(x,t)=\tilde{u}^{\infty}(x)+c_{\infty}\cdot t$ in $ C^{4+\gamma}(\overline{\Omega})$ as $t \to +\infty$ for any $ 0<\gamma<\alpha.$ That is,
	$$\lim_{t\to+\infty}\|u(\cdot,t)-u^{\infty}(\cdot,t)\|_{C^{4+\gamma}(\overline{\Omega})}=0.$$
	And $\tilde{u}^\infty(x)\in C^{4+\alpha}(\overline{\Omega})$ is a solution of
	\begin{equation}\label{trans solu}
		\left.\left\{\begin{array}{l}
			\sum\limits_{i=1}^{n}\arctan \left(\lambda_{i}(D^2u)\right)=f(x,Du)+c_\infty,\quad x\in\Omega,\\Du(\Omega)=\tilde{\Omega}.
		\end{array}\right.\right.
	\end{equation}
	The constant $c_\infty$ depends only on $n$, the diameters and volumes of $\Omega, \tilde{\Omega}$, the upper and lower boundness of principal curvature of $\partial\Omega, \partial\tilde\Omega$,  $\|u_0\|_{C^{2+\alpha}(\overline{\Omega})}$ and $\|f\|_{ C^{2+\alpha}(\overline{\Omega}\times\overline{\tilde{\Omega}})}$. The solution to \eqref{trans solu} is
	unique up to additions of constants.
	Especially, if $f$ and $u_0$ are smooth, then
	$\tilde{u}^\infty\in C^\infty(\overline{\Omega})$.
\end{thm}
\begin{rem}\label{Adelta}
	Let $f(x,p)=\iota\cdot p+\kappa\cdot x+c$,  then $f\in \mathscr{A}$.
\end{rem}

We now briefly review some relevant work on translating solutions of nonlinear equations. Zhou \cite{MR4759606} studied the translating mean curvature equations and established general existence results by introducing a non--closed--minimal (NCM) condition on the underlying domain. General capillary--type boundary conditions and existence theorems can be found in \cite{MR3985383}. For the mean curvature flow equation, Gui--Ju--Jian--Lu \cite{Ju2010963,Gui2010441,Jian20113967} derived existence, uniqueness results, gradient estimates, Liouville type theorems, and other properties in certain special cases.  Ju--Bao--Jian \cite{MR2901343} investigated translating solutions of Gauss curvature flow on exterior domains and proved the existence of viscosity solutions to a class of Monge--Amp\`ere equations. Choi--Daskalopoulos \cite{Choi2024} got the classication of ancient solutions to the Gauss curvature flow under the assumption that the solutions are contained in a cylinder of bounded cross--section. 

For translating solitons, Spruck--Sun \cite{MR4236553} proved that any complete, immersed, globally orientable, uniformly 2--convex translating soliton for the mean curvature flow is locally strictly convex. Li \cite{MR3824855} studied translating mean curvature flow of hypersurfaces in \(\mathbb{R}^{n+1}\) and established the global existence of the flow as well as its convergence properties. Rafael \cite{MR3907583} and Santaella--Jose \cite{MR4398426} investigated a special translating soliton equation and proved the existence of a solution on a strip of \(\mathbb{R}^2\) using the Perron method and convexity estimates for the smallest principal curvature, respectively. Lon--Yuan \cite{MR4612703} examined the case where the velocity involves a positive power of the mean curvature and a driving force, and they established global existence and convergence results for the translating solution of the initial boundary value problem in a cylinder. Gao--Li--Wu \cite{MR3282644} demonstrated that when the boundary manifold is a convex cylinder, the mean curvature flow converges to a spacelike hypersurface moving at a constant speed. For blow--up scenarios, James--Wu--Zhang \cite{MR4078822} showed that the flow of noncompact hypersurfaces converges to a translating soliton known as the bowl soliton.

The remainder of this article is organized as follows. To prove Theorem \ref{main-thm}, we first establish the short time existence of the parabolic flow in Section \ref{Section2}. By the second boundary value condition \eqref{BC}, the gradient \( Du \) is bounded. Consequently, Sections \ref{section3} and \ref{sec4} are devoted to deriving the uniformly oblique estimate and the \( C^2 \) estimate. Finally, in Section \ref{section5}, we present the long time existence and convergence results for the parabolic flow, and Theorem \ref{prethm} follows as a corollary of Theorem \ref{main-thm}.

Throughout the following, Einstein's convention of summation over repeated indices
will be adopted. We denote, for a smooth function $f(x,p)$,
$$f_{i}=\frac{\partial f}{\partial x_{i}},\quad f_{p_i}=\frac{\partial f}{\partial p_{i}},\quad f_{ij}=\frac{\partial^{2}f}{\partial x_{i}\partial x_{j}},\quad f_{p_ip_j}=\frac{\partial^{2}f}{\partial p_{i}\partial p_{j}},\quad f_{ijk}=\frac{\partial^{3}f}{\partial x_{i}\partial x_{j}\partial x_{k}},\ldots.$$

\section{Short time existence of the parabolic flow}\label{Section2}
By the methods in the second boundary value problems for equations of Monge--Amp\`ere type \cite{Urbas1997OnTS}, the second boundary condition in \eqref{BC} can be reformulated as
$$h(Du)=0,\quad \quad x\in\partial\Omega,\quad t>0.$$
where we need
\begin{defn}\label{defh}
	A smooth function $h:\mathbb{R}^n\to\mathbb{R}$ is called the defining function of $\tilde{\Omega}$ if
	$$\tilde{\Omega}=\{p\in\mathbb{R}^n:h(p)>0\},\quad|Dh|=1~on~\partial\tilde{\Omega},\quad -\Theta I \leq D^2h\leq-\theta I~in~\tilde{\Omega},$$
	where $\Theta>\theta>0$ are constants depending only on $\partial\tilde{\Omega}$.
\end{defn}
For the convenience, we denote $\beta=(\beta^1,\ldots,\beta^n)$ with $\beta^i:=h_{p_i}(Du)$, and $\nu=$ $(\nu_1,\ldots,\nu_n)$ as the unit inward normal vector at $x\in\partial\Omega$. 
The expression of the inner product is
$$\langle\beta,\nu\rangle=\beta^i\nu_i.$$
According to the proof in \cite{Urbas1997OnTS}, we have the strictly oblique boundary condition.
\begin{lem}\label{Urbas}
	If $u\in C^2(\overline{\Omega})$ with $D^2u>0$, then there holds $ \langle\beta,\nu\rangle>0$.
\end{lem}

For $\lambda=\left(\lambda_1,\cdots,\lambda_n\right)$, we denote
$$F(\lambda):=\sum_{i=1}^{n}\arctan\lambda_{i},$$   
it is obvious that $F$ is a smooth symmetric function defined on $\Gamma_{n}^{+}$, where
$$\Gamma_{n}^{+}:=\{\left(\lambda_1,\cdots,\lambda_n\right)\in \mathbb{R}^n:\lambda_{i}>0,i=1,\ldots,n\}. $$
By direct calculation, we get
\begin{equation}\label{bound}
	F\left(0,\ldots,0\right)=0,\quad F\left(\infty,\ldots,\infty\right)=\frac{n\pi}{2},
\end{equation}
\begin{equation}\label{posdef}
	\frac{\partial F}{\partial\lambda_i}=\frac{1}{1+\lambda_{i}^2}>0,\quad 1\leq i\leq n \quad \text{on} \quad \Gamma_{n}^{+},
\end{equation}
and 
\begin{equation}\label{concave}
	\left(\frac{\partial^2 F}{\partial\lambda_i\partial\lambda_j}\right)=\left(\frac{-2\delta_{ij}\lambda_j}{\left(1+\lambda_i^2\right)^2}\right)\leq 0\quad \text{on} \quad \Gamma_{n}^{+}. 
\end{equation}
For a positive definite symmetric \( n \times n \) matrix \( A = (a_{ij}) \) with eigenvalues \( \lambda_1(A), \cdots, \lambda_n(A) \), the function \( F \) can also be expressed as \( F[A] \), where the independent variables are the entries \( a_{ij} \). We denote
$$ F^{ij}:=\frac{\partial F}{\partial a_{ij}},\quad F^{ij,rs}:=\frac{\partial^2F}{\partial a_{ij}\partial a_{rs}}.$$
Thus, the parabolic flow \eqref{EQP}--\eqref{IC} is equivalent
to the evolution problem
\begin{equation}\label{evoproblem}
	\left\{
	\begin{array}{lll}
		\dfrac{\partial u}{\partial t}=F[D^2u]-f(x,Du),&x\in\Omega,&t>0,\\
		h(Du)=0,&x\in\partial\Omega,&t>0,\\
		u=u_0,&x\in\overline{\Omega},&t=0.
	\end{array}
	\right.
\end{equation}

To establish the short time existence of classical solutions of \eqref{evoproblem}, we use the inverse
function theorem in Fréchet spaces and the theory of linear parabolic equations for
oblique boundary condition. Our approach is inspired by the methodology used in proving the short time existence of convex solutions for the second boundary value problem in the Lagrangian mean curvature flow, as detailed in \cite{Huang2014OnTS}. We include the details for the convenience of the readers.

\begin{lem}\label{inv2}
	(See Theorem 2 in \cite{Eke2011}) Let $X$ and $Y$ be Banach spaces with the norms $\|\cdot\|_X$ and $\|\cdot\|_Y$, respectively. Suppose
	$$J:X\to Y$$
	is continuous and Gâteaux--differentiable, with $J[v_{0}] = w_{0}$. Assume that the derivative $DJ[v]$ has a right inverse $L[v]$, uniformly bounded about $v$ in a neighborhood of $v_0$. That is, for any $y\in Y$,
	$$DJ[v]L[v]y=y,$$
	and there exist $R>0 $ and $m>0$, if $\|v-v_0\|_X\leq R$, we have
	$$\|L[v]\|\leq m,$$
	where $\|L[v]\|:=\sup_{\|y\|_Y\leq1}\|L[v](y)\|_X$. Then, for every $w \in Y$ , if
	$$\|w-w_0\|_Y<\frac Rm,$$
	there is some $v \in X$ such that
	$$\|v-v_0\|_X<R,$$
	and
	$$J(v)=w.$$
\end{lem}
%

Now, we can prove the short time existence of solutions of \eqref{evoproblem}. In the following, we denote
$\Omega_{T}:=\Omega\times(0,T]$ to be the parabolic cylinder, and $\overline{\Omega_{T}}=\overline{\Omega}\times[0,T]$ is the closure of $\Omega_{T}$ with  parabolic boundary $\partial_p\Omega_{T}:=\overline{\Omega_T}\setminus \Omega_T$.
\begin{prop}\label{shorttime}
	Let $f \in \mathscr{A}$, there exist $T_{\max}>0$ and $u\in C^{2+\alpha_0,\frac{2+\alpha_0}2}(\overline{\Omega}_{T_{\max}})$ for some $0<\alpha_0<\alpha$, such that $u$ is a unique solution of \eqref{evoproblem} and is strictly convex in $x$ variable, where $T_{\max}$  depends only on $\Omega,\tilde{\Omega},\|u_0\|_{{C^{2+\alpha}(\overline{\Omega})}}, \|f\|_{{C^{\alpha}(\overline{\Omega}\times\overline{\tilde{\Omega}})}}$. 
\end{prop}
\begin{proof}
	For $T>0,0<\alpha_0<\alpha$, denote the Banach spaces
	$$X=\left\{u(x,t)\in C^{2+\alpha_0,1+\frac{\alpha_0}{2}}(\overline{\Omega_T})|u~is~strictly~convex~in~x \right\},$$ 
	and
	$$Y=C^{\alpha_0,\frac{\alpha_0}{2}}(\overline{\Omega_T})\times C^{1+\alpha_0,\frac{1+\alpha_0}2}(\partial\Omega\times[0,T])\times C^{2+\alpha_0}(\overline{\Omega}),$$
	where
	$\|\cdot\|_{X}=\|\cdot\|_{C^{2+\alpha_0,1+\frac{\alpha_0}{2}}(\overline{\Omega_T})}$ and $$\|\cdot\|_{Y}=\|\cdot\|_{{C^{\alpha_0,\frac{\alpha_0}{2}}(\overline{\Omega_{T}})}}+\|\cdot\|_{{C^{1+\alpha_0,\frac{1+\alpha_0}{2}}(\partial\Omega\times[0,T])}}+\|\cdot\|_{{C^{2+\alpha_0}(\overline{\Omega})}}.$$
	Define a map $J: X\to Y$ by
	\begin{equation*}
		J[v]=\left\{
		\begin{array}{lll}
			\dfrac{\partial v}{\partial t}-F[D^2v]+f(x,Dv),&x\in\Omega,&T\geq t>0,\\
			h(Dv),&x\in\partial\Omega,&T\geq t>0,\\
			v,&x\in\overline{\Omega},&t=0.
		\end{array}
		\right.
	\end{equation*}
	Thus, the strategy is to use Lemma \ref{inv2} to obtain the short time
	existence result.

	By Theorems 8.8, 8.9 in \cite{lieberman} and Lemma \ref{Urbas}, there exists $T_1 > 0$, which depends only on $\Omega,\tilde{\Omega},\|u_0\|_{{C^{2+\alpha}(\overline{\Omega})}}$ and $ \|f\|_{{C^{\alpha}(\overline{\Omega}\times\overline{\tilde{\Omega}})}}$ such that we can find
	$$v_0\in C^{2+\alpha,1+\frac{\alpha}2}(\overline{\Omega}_{T_1})$$
	to be strictly convex in $x$ variable (obviously $v_0\in X$), which satisfies the following problems:
	\begin{equation}\label{eq2.6}
		\left\{
		\begin{array}{lll}
			\dfrac{\partial v_0}{\partial t}-\Delta v_0+f(x,Dv_0) =F[D^2u_0]-\Delta u_0,&x\in\Omega,&T_1\geq t>0,\\
			h_p(Du_0)\cdot Dv_0=h(Du_0)+h_p(Du_0)\cdot Du_0,&x\in\Omega,&T_1\geq t>0,\\
			v_0=u_0,&x\in\overline{\Omega},&t=0.
		\end{array}
		\right.
	\end{equation}
	
	We want to get the solutions $u$ from $v_0$, let
	$$f_0:=\frac{\partial v_0}{\partial t}-F[D^2v_0]+f(x,Dv_0),\quad w_0:=(f_0,h(Dv_0),u_0),$$
	then we can show that 
	$$J[v_0]=w_0$$
	and by \eqref{eq2.6} and  internal interpolation inequality, it turns out that
	\begin{equation*}
		\begin{aligned}
			\|f_0\|_{C^{\alpha_0,\frac{\alpha_0}{2}}(\overline{\Omega}_{T_{1}})}& =\|\Delta v_0-\Delta u_0+F[D^2u_0]-F[D^2v_0]\|_{C^{\alpha_0,\frac{\alpha_0}{2}}(\overline{\Omega}_{T_1})} \\
			&\leq\|\Delta v_0-\Delta u_{0}\|_{{C^{{\alpha_0,\frac{\alpha_0}{2}}}(\overline{\Omega}_{{T_{1}}})}}+\|F[D^{2}u_{0}]-F[D^{2}v_0]\|_{{C^{{\alpha_0,\frac{\alpha_0}{2}}}(\overline{\Omega}_{{T_{1}}})}} \\
			&\leq C\|D^2v_0-D^2u_0\|_{C^{\alpha_0,\frac{\alpha_0}{2}}(\overline{\Omega}_{T_1})},
		\end{aligned}
	\end{equation*}
	and 
	\begin{equation*}
		\begin{aligned}
			&\quad \|h(Dv_0)\|_{{C^{1+\alpha_0,\frac{1+\alpha_0}{2}}(\partial\Omega\times[0,T_1])}}\\
			&=\|h(v_0)-h(Du_0)+h_p(Du_0)\cdot D(u_0-v_0)\|_{{C^{1+\alpha_0,\frac{1+\alpha_0}{2}}(\partial\Omega\times[0,T_1])}}\\
			&\leq \|h(Dv_0)-h(Du_0)\|_{{C^{1+\alpha_0,\frac{1+\alpha_0}{2}}(\partial\Omega\times[0,T_1])}}\\
			&~~~+\|h_p(Du_0)\cdot D(u_0-v_0)\|_{{C^{1+\alpha_0,\frac{1+\alpha_0}{2}}(\partial\Omega\times[0,T_1])}}\\
			&\leq C\|Dv_0-Du_0\|_{{C^{1+\alpha_0,\frac{1+\alpha_0}{2}}(\partial\Omega\times[0,T_1])}}\\
			&\leq C\left(\|D^2v_0-D^2u_0\|_{C^{\alpha_0,\frac{\alpha_0}{2}}(\overline{\Omega_{T_1}})}+\|Dv_0-Du_0\|_{C(\overline{\Omega_{T_1}})}\right),
		\end{aligned}
	\end{equation*}
	where $C$ is a constant depending only on $n,\text{diam}(\Omega),|\Omega|,\theta$ and $\Theta$.
	
	Since $v_0\in C^{2+\alpha,\frac{2+\alpha}2}(\overline{\Omega_{T_1}})$ and $v_0(x,t)|_{t=0}=u_0(x)$, we obtain
	\begin{equation}\label{lim2.3}
		\lim_{t\to0}\|v_0(\cdot,t)-u_{0}\|_{C^{2}(\overline{\Omega})}=0\quad and \quad \lim_{t\to0}\|D^{2}v_0(\cdot,t)-D^{2}u_{0}(\cdot)\|_{C^{\alpha_0}(\overline{\Omega})}=0.
	\end{equation}
	For any $\alpha_0<\alpha$, we have
	\begin{align*}
		&\frac{|(D^{2}v_0(x,t)-D^{2}u_{0}(x))-(D^{2}v_0(y,\tau)-D^{2}u_{0}(y))|}{|x-y|^{\alpha_0}+|t-\tau|^{{\frac{\alpha_0}{2}}}} \\
		\leq&\frac{|(D^2v_0(x,t)-D^2u_0(x))-(D^2v_0(y,t)-D^2u_0(y))|}{|x-y|^{\alpha_0}} \\
		&+|t-\tau|^{\frac{\alpha-\alpha_0}2}\frac{|(D^2v_0(y,t)-D^2u_0(y))-(D^2v_0(y,\tau)-D^2u_0(y))|}{|t-\tau|^\frac{\alpha}2}. 
	\end{align*}
	Then, we get
	\begin{equation}\label{eq2.4}
		\begin{aligned}
			\|D^{2}v_0-D^{2}u_{0}\|_{C^{\alpha_0,\frac{\alpha_0}{2}}(\overline{\Omega_{T_1}})}&\leq\max_{0\leq t\leq T_1}\|D^{2}v_0(\cdot,t)-D^{2}u_{0}(\cdot)\|_{C^{\alpha_0}(\overline{\Omega})}\\
			&\quad +T_1^{\frac{\alpha-\alpha_0}2}\|D^2v_0-D^2u_0\|_{C^{\alpha,\frac{\alpha}2}(\overline{\Omega_{T_1}})}.
		\end{aligned}
	\end{equation}
	Combining \eqref{lim2.3} with \eqref{eq2.4}, we obtain
	\begin{equation}\label{lim2.5}
		\lim_{T_{1}\to0}\|v_0-u_0\|_{C^{2+\alpha_0,\frac{2+\alpha_0}{2}}(\overline{\Omega_{T_1}})}=0.
	\end{equation}
	Hence, we conclude for any given positive constants $R$ and $m$, 
	there exists $T_{\max}>0\stackrel{\cdot}{(T_{\max}\leq T_1)}$ to be small enough such that
	$$\|f_0\|_{C^{\alpha_0,\frac{\alpha_0}{2}}(\overline{\Omega_{T_{\max}}})}+\|h(Dv_0)\|_{{C^{1+\alpha_0,\frac{1+\alpha_0}{2}}(\partial\Omega\times[0,T_{\max}])}}<\frac Rm,$$
	Therefore, for $T=T_{\max}$ and $w=(0,0,u_0)$,
	$$\|w-w_0\|_Y=\|0-f_0\|_{C^{\alpha_0,\frac{\alpha_0}2}(\overline{\Omega_{T_{\max}}})}+\|0-h(Dv_0)\|_{{C^{1+\alpha_0,\frac{1+\alpha_0}{2}}(\partial\Omega\times[0,T_1])}}<\frac Rm.$$

	Next, we show $DJ[v]$ has a uniformly bounded right inverse $L[v]$ near $v_0$. By density, for the above $R>0$,  there exists $v\in X$ such that
	$$\|v-v_0\|_{C^{2+\alpha_0,\frac{2+\alpha_0}{2}}(\overline{\Omega_{T_1}})}<R.$$
	The computation of the Gâteaux derivative shows that for any $u\in X$,
	\begin{equation*}
		DJ[v](u)=
		\left\{
		\begin{array}{lll}
			\dfrac{\partial u}{\partial t}-F^{ij}[D^2v]u_{ij}+f_{p_i}(x,Dv) u_i,&x\in\Omega,&T\geq t>0,\\
			h_{p_i}(Dv)u_i,&x\in\partial\Omega,&T\geq t>0,\\
			u,&x\in\overline{\Omega},&t=0.
		\end{array}
		\right.
	\end{equation*}	
	For each $y:=(\overline{f},\overline{g},\overline{w})\in Y$, $\|y\|_Y \leq 1$ and the above $m>0$, using the Theorems 8.8, 8.9 in \cite{lieberman}, we know that there exists a unique $u\in X(T=T_{\max})$ satisfying  $DJ[v](u)=y$, that is
	\begin{equation*}
		\left\{
		\begin{array}{lll}
			\dfrac{\partial u}{\partial t}-F^{ij}[D^2v]u_{ij}+f_{p_i}(x,Dv) u_i=\overline{f},& x\in\Omega,&T_{\max}\geq t>0,\\
			h_{p_i}(Dv)u_i=\overline{g},& x\in\partial\Omega,&T_{\max}\geq t>0,\\
			u=\overline{w},&x\in\overline{\Omega},&t=0,
		\end{array}
		\right.
	\end{equation*}
	and 
	\begin{align*}
		&\quad \|u\|_{C^{2+\alpha_0,\frac{2+\alpha_0}2}(\overline{\Omega_{T_{\max}}})}\\
		&\leq m\left(\|\overline{f}\|_{C^{\alpha_0,\frac{\alpha_0}{2}}(\overline{\Omega_{T_{\max}}})}+\|\overline{g}\|_{C^{1+\alpha_0,\frac{1+\alpha_0}2}(\partial\Omega\times[0,T_{\max}])}+\|\overline{w}\|_{C^{2+\alpha_0}(\overline{\Omega})}\right).
	\end{align*}
	For $T = T_{\max}$, by the definition of the Banach spaces $X$ and $Y$, we can rewrite the above
	Schauder estimates as
	$$\|u\|_X\leq m.$$
	It means that the derivative $DJ[v](u) = y$ has a right inverse $u = L[v](y)$ and
	$$\|L[v]\|=\sup_{\|y\|_Y\leq1}\|L[v](y)\|_X\leq m.$$
	Thus, By Lemma \ref{inv2}, we get the short time existence of solutions of \eqref{evoproblem}.
	By the strong maximum principle, the strictly convex solution to \eqref{evoproblem} is unique.
\end{proof}

\section{The uniformly obliqueness estimate}\label{section3}
In this section, we turn to establish the uniformly obliqueness estimate.
This treatment is similar to the
problems presented in \cite{Urbas1997OnTS,Schnrer2003NeumannAS,Huang2014OnTS,Wang2024}, but requires some modification to accommodate
the more general situation. Specifically, the structure conditions of $F$ are
needed in order to derive differential inequalities from barriers which can be used.

By Proposition \ref{shorttime} and the regularity theory of parabolic equations, we may assume that $u$ is a strictly convex solution of \eqref{EQP}--\eqref{IC}  in the class $C^{2+\alpha_0,1+\frac{\alpha_0}{2}}(\overline{\Omega_T})\cap$ $C^{4+\alpha_0,2+\frac{\alpha_0}{2}}(\Omega_T)$ for some $T=T_{\max}>0$ and $0<\alpha_0<\alpha$.

\begin{lem}\label{utestimate}
	($\dot{u}$-estimates) If the convex solution $u$ to \eqref{EQP}--\eqref{IC} exists and $f\in \mathscr{A}, $, then 
	$$\min_{\overline{\Omega}}F[D^2u_0]-\max_{\overline{\Omega}}f(x,Du_0)\leq\dot{u}\leq\max_{\overline{\Omega}}F[D^2u_0]-\min_{\overline{\Omega}}f(x,Du_0),$$
	where $\dot{u}:=\frac{\partial u}{\partial t}$.
\end{lem}
\begin{proof}
	From \eqref{EQP}, a direct computation shows that
	$$\frac{\partial(\dot{u})}{\partial t}-F^{ij}\partial_{ij}(\dot{u})+f_{p_i} \partial_i(\dot{u})=0,\quad in~\Omega_{T}.$$
	Using the maximum principle, we see that
	$$\min_{\overline{\Omega_T}}(\dot{u})=\min_{\partial_p\Omega_T}(\dot{u}).$$
	
	Suppose that \(\dot{u}(x_0, t_0) = \min_{\overline{\Omega_T}} (\dot{u})\). We considering the following two cases based on the location of \((x_0, t_0)\) with $t_0>0$:
	
	\begin{enumerate}
		\item 
		If the minimum of \(\dot{u}\) is attained at an interior point \((x_0, t_0) \in \Omega_T\), then by the strong maximum principle for parabolic equations, $u(x,t)=u_0(x)+\dot{u}(x_0, t_0)t$ throughout \(\overline{\Omega}\times[0,t_0]\).  On the other hand, by equation \eqref{EQP},  $\dot{u}(x_0, t_0)=F[D^2u_0(x_0)]-f(x_0,Du_0(x_0))$, and the result follows trivially in this case.
		
		\item 
		If the minimum of \(\dot{u}\) is is not attained at an interior point, then \(x_0 \in \partial\Omega\), we invoke Lemma \ref{Urbas} and the Hopf Lemma for parabolic equations (cf. \cite{Li2005AGP,1d5f094ab89043fd80de075be163feaf}). These results imply that the following condition must hold:
		\[
		\frac{\partial \dot{u}}{\partial \nu}(x_0, t_0) > 0.
		\]
		On the other hand, we differentiate the boundary condition and then obtain
		$$\dot{u}_\beta=h_{p_k}(Du)\frac{\partial\dot{u}}{\partial x_k}=\frac{\partial h(Du)}{\partial t}=0,\quad on~\partial\Omega\times(0,T],$$
		it is a contradiction.
	\end{enumerate}
	
	So we deduce that
	\begin{align*}
		\dot{u}\geq\min_{\overline{\Omega_T}}(\dot{u})&=\min_{\overline{\Omega}}\left(F[D^2u_0]-f(x,Du_0)\right)\\
		&\geq\min_{\overline{\Omega}}F[D^2u_0]-\max_{\overline{\Omega}}f(x,Du_0).
	\end{align*}
	For the same reason, we have
	\begin{align*}
		\dot{u}\leq\max_{\overline{\Omega_T}}(\dot{u})
		&\leq\max_{\overline{\Omega}}F[D^2u_0]-\min_{\overline{\Omega}}f(x,Du_0).
	\end{align*}
	Putting these facts together, the assertion follows.
\end{proof}
For the convenience, we introduce the set
$$\mathscr{A}_{\delta}:=\left\{f\in \mathscr{A}|\max_{x\in\overline{\Omega},p,q\in \overline{\tilde{\Omega}}}|f(x,p)-f(x,q)|\leq \delta\right\},$$
where $\delta$ is any positive constant satisfying
$$\delta<\min\{\frac{n\pi}{2}-\max_{\overline{\Omega}}F[D^{2}u_{0}],\min_{\overline{\Omega}}F[D^{2}u_{0}]\}.$$

\begin{lem}\label{lm3.2}
	Suppose  $\lambda _1( x, t) , \ldots , \lambda _n( x, t)$ be the eigenvalues of $D^2u$ at $(x,t)\in\Omega_T$. Let $f\in \mathscr{A}_{\delta}$ and \eqref{bound}, \eqref{posdef} hold,
	if $u$ is a strictly convex solution to \eqref{EQP}--\eqref{IC}, then there exist $	\Lambda_{1}>0$  depending only on $n$, $\max_{\overline{\Omega}}F[D^{2}u_{0}]$ and $\delta$, such that $F$ satisfes the structure conditions:
	\begin{equation}\label{cond1}
		\Lambda_{1}\leq \sum_{i=1}^n\frac{\partial F}{\partial\lambda_i}< n,
	\end{equation}
	and
	\begin{equation}\label{cond2}
		\Lambda_{1}\leq\sum_{i=1}^n\frac{\partial F}{\partial\lambda_i}\lambda_i^2< n.
	\end{equation}
\end{lem}
\begin{proof}
	By \eqref{posdef} and Lemma \ref{utestimate}, we obtain
	\begin{align*}
		F\left(\operatorname*{min}_{1\leq i\leq n}\lambda_{i}(x,t),\ldots,\operatorname*{min}_{1\leq i\leq n}\lambda_{i}(x,t)\right)& \leq F[D^2u]=\dot{u}+f(x,Du) \\
		&\leq\max_{\overline{\Omega}}F[D^2u_0]-\min_{\overline{\Omega}}f(x,Du_0)+f(x,Du) \\
		&\leq\max_{\overline{\Omega}}F[D^2u_0]+\delta \\
		&<\frac{n\pi}{2},
	\end{align*}
	and
	\begin{align*}
		F\left(\operatorname*{max}_{1\leq i\leq n}\lambda_{i}(x,t),\ldots,\operatorname*{max}_{1\leq i\leq n}\lambda_{i}(x,t)\right)& \geq F[D^2u]=\dot{u}+f(x,Du) \\
		&\geq\min_{\overline{\Omega}}F[D^2u_0] -\max_{\overline{\Omega}}f(x,Du_0)+f(x,Du) \\
		&\geq\max_{\overline{\Omega}}F[D^2u_0]-\delta \\
		&>0.
	\end{align*}
	By the monotonicity of $F$ and \eqref{bound}, there exist positive constants $\mu_1, \mu_2$ depending only on $n$ and $max_{\overline{\Omega}}F[D^2u_0]$ such that
	$$\min_{1\leq i\leq n}\lambda_i(x,t)\leq\mu_1,\quad\max_{1\leq i\leq n}\lambda_i(x,t)\geq\mu_2.$$ 
	Then, 
	\begin{equation*}\label{cond1*}
		\frac{1}{1+\mu_1^2}\leq \sum_{i=1}^n\frac{\partial F}{\partial\lambda_i}< n  ,
	\end{equation*}
	and
	\begin{equation*}\label{cond2*}
		\frac{\mu_2^2}{1+\mu_2^2}\leq \sum_{i=1}^n\frac{\partial F}{\partial\lambda_i}\lambda_i^2<n,
	\end{equation*}
	thus, we get the desired result.
\end{proof}


For technical needs below, we introduce the Legendre transformation of $u.$ For any
$(x,t)\in\Omega_T$, define
$$\tilde{x}:=Du(x,t),$$
and
$$\tilde{u}(\tilde{x},t):=x\cdot Du(x,t) -u(x,t).$$
In terms of $\tilde{x}$ and $\tilde{u}(\tilde{x},t)$, we can easily check that
$$D_{\tilde{x}}\tilde{u}=x,\quad \tilde{u}_t=-u_t$$
and
$$\left(\frac{\partial^2\tilde{u}}{\partial\tilde{x}_i \partial\tilde{x}_j}\right)=\left(\frac{\partial^2u}{\partial x_i \partial x_j}\right)^{-1}.$$
For any $\left(\lambda_1,\ldots,\lambda_n\right)\in \Gamma_{n}^{+}$, let $\left(\mu_1,\ldots,\mu_n\right)$ be the eigenvalues of $D^2\tilde{u}$ at $\tilde{x}=Du(x)$, then
$$\mu_i=\lambda_i^{-1},\quad i=1,2,\ldots,n.$$
Denote
$$\tilde{F}\left(\mu_1,\ldots,\mu_n\right):=\dfrac{n\pi}{2}-\sum_{i=1}^n\arctan\frac{1}{\mu_i},$$
we obtain
\begin{equation}\label{posdefmu}
	\frac{\partial \tilde{F}}{\partial\mu_i}=\frac{1}{1+\mu_{i}^2}>0,\quad 1\leq i\leq n \quad \text{on} \quad \Gamma_{n}^{+},
\end{equation}
and 
\begin{equation}\label{concavemu}
	\left(\frac{\partial^2 \tilde{F}}{\partial\mu_i\partial\mu_j}\right)=\left(\frac{-2\delta_{ij}\mu_j}{\left(1+\mu_i^2\right)^2}\right)\leq 0\quad \text{on} \quad \Gamma_{n}^{+}. 
\end{equation}
By means of Legendre transformation, we can also define a smooth function $\tilde{h}$ as the defining function of  $\Omega$. That is
$$\Omega=\{\tilde{p}\in\mathbb{R}^n:\tilde{h}(\tilde{p})>0\},\quad|D\tilde{h}|=1~on~\partial\Omega,\quad -\tilde{\Theta} I \leq D^2\tilde{h}\leq-\tilde{\theta}I, ~in~\Omega.$$
where $\tilde{\theta},\tilde{\Theta}$ are positive constants depending only on $\partial\Omega$.

Moreover, it follows from the defination of $\tilde{h}$ and \eqref{evoproblem} that

\begin{equation}\label{tilde u eq1}
	\left\{
	\begin{array}{lll}
		\dfrac{\partial\tilde{u}}{\partial t}=\tilde{F}[D^2\tilde{u}]+f(D\tilde{u},\tilde{x})-\dfrac{n\pi}{2},&\tilde{x}\in\tilde{\Omega},&t>0,\\
		\tilde{h}(D\tilde{u})=0,&\tilde{x}\in\partial\tilde{\Omega},&t>0,\\
		\tilde{u}=\tilde{u}_0,&\tilde{x}\in\overline{\tilde{\Omega}},&t=0,
	\end{array}
	\right.
\end{equation}
where $\tilde{u}_0$ is the Legendre transformation of $u_0$.  
\begin{defn}
	We say that $\tilde{u}$ in \eqref{tilde u eq1} is a dual solution $u$ to \eqref{evoproblem}.
\end{defn}

\begin{rem}
	By Lemma \ref{lm3.2} and \eqref{posdefmu}, if $u$ is a strictly convex solution to \eqref{EQP}--\eqref{IC}, then  $\tilde{F}$ also satisfes \eqref{cond1} and \eqref{cond2}.
\end{rem}
In order to establish the $C^{2}$ estimates, we first need to get the uniformly obliqueness estimates, a parabolic version of a result of Urbas \cite{Urbas1997OnTS}, which was given in \cite{Schnrer2003NeumannAS}. Recall the Lemma \ref{Urbas}, we get a strict positive lower bound of the quantity $\operatorname*{inf}_{\partial\Omega}\langle\beta,\nu\rangle$ which does not depend on $t$.

We firstly prove the following key lemma. Let the diameter and volume of the domain \(\Omega\) be denoted by \(\text{diam}(\Omega)\) and \(|\Omega|\), respectively.

\begin{lem}\label{vnestimate}
	Let $v:=\langle\beta,\nu\rangle+h(Du)$, suppose $v$ reaches the minimum on $\partial\Omega\times[0,T]$ at $(x_0,t_0)$.  If $f\in \mathscr{A}_\delta$, satisfying 
	\begin{equation}\label{lcond}
		|D_pf|\leq\frac{\tilde{\theta} \Lambda_1}{2}\cdot\frac{1}{\max_{\overline{\Omega}}|D\tilde{h}|},
	\end{equation}
	and $u$ is a strictly convex solution to \eqref{EQP}--\eqref{IC}, then
	\begin{equation}\label{vn}
		\frac{\partial v}{\partial \nu}(x_0,t_0)\geq-C_1,
	\end{equation}
	where $C_1$ is a constant depending only on $n, \text{diam}(\Omega),  |\Omega|, \theta,\Theta,\tilde{\theta}$,  $\max_{\overline{\Omega}}F[D^{2}u_{0}]$ and $\delta$.
\end{lem}
\begin{proof}
	By rotation, we may assume that $x_0=0,t_0=1$ and $\nu(x_0,t_0)=(0,1)=:e_n$. Denote a neighborhood of $x_0$ in $\Omega$ by
	$$\mathscr{N}_\rho:=\Omega\cap B_\rho(x_0),$$
	where $\rho$ is a positive constant. To obtain \eqref{vn}, we need to consider the function
	$$\Phi(x,t):=v(x,t)-v(x_0,t_0)+C_0\tilde{h}(x)+A|x-x_0|^2,$$
	where $C_{0}$ and $A$ are positive constants to be determined.
	
	We firstly estimate the $Lv$, where $L:=F^{ij}\partial_{ij}-f_{p_i}\partial_i-\partial_t$.  By the definition of $\tilde h$, it follows that
	\begin{equation}\label{Ltildeh}
		L\tilde{h}\leq-\tilde{\theta}\sum_{i=1}^nF^{ii}-f_{p_i}\partial_i\tilde{h}.
	\end{equation}
	On the other hand,
	\begin{equation}\label{Lv}
		\begin{aligned}
			Lv&=h_{p_{k}p_{l}p_{m}}\nu_{k}F^{ij}u_{li}u_{mj}+2h_{p_{k}p_{l}}F^{ij}\nu_{kj}u_{li} +h_{p_{k}p_{l}}F^{ij}u_{lj}u_{ki}\\
			&\quad +h_{p_{k}p_{l}}\nu_{k}Lu_{l}+h_{p_{k}}L\nu_{k}+h_{p_{k}}Lu_{k}.
		\end{aligned}
	\end{equation}
	Now, we estimate the right hand side of \eqref{Lv}. To do so, we fix a point $p$ and rotate coordinates
	so that $D^2u$ is diagonalized at $p$.
	By \eqref{cond2}, we have
	$$|h_{p_kp_lp_m}\nu_kF^{ij}u_{li}u_{mj}|\leq C\sum_{i=1}^n\frac{\partial F}{\partial\lambda_i}\lambda_i^2\leq C_2,$$
	and 
	$$|h_{p_kp_l}F^{ij}u_{lj}u_{ki}|\leq C\sum_{i=1}^n\frac{\partial F}{\partial\lambda_i}\lambda_i^2\leq C_3.$$
	where $C_2,C_3$ are constants depending only on $n,\theta$ and $\Theta$. For the second term, by Cauchy inequality, we obtain
	$$\begin{aligned}
		|h_{ p_{k} p_{l}}F^{ij}\nu_{kj}u_{li}|& \leq C\sum_{i=1}^{n}\frac{\partial F}{\partial\lambda_{i}}\lambda_{i}=C\sum_{i=1}^{n}\sqrt{\frac{\partial F}{\partial\lambda_{i}}}\sqrt{\frac{\partial F}{\partial\lambda_{i}}}\lambda_{i} \\
		&\leq C\left(\sum_{i=1}^n\frac{\partial F}{\partial\lambda_i}\right)^{\frac12}\left(\sum_{i=1}^n\frac{\partial F}{\partial\lambda_i}\lambda_i^2\right)^{\frac12} \\
		&\leq C_{4}.
	\end{aligned}$$
	By \eqref{EQP}, we have $Lu_l = f_l$. Then, we get
	$$|h_{{p_{k}p_{l}}}\nu_{k}Lu_{l}|\leq C_{5},\quad|h_{{p_{k}}}Lu_{k}|\leq C_{6}.$$
	It follows from \eqref{cond1} that
	$$|h_{p_k}L\nu_k|\leq C_{7}\sum_{i=1}^nF^{ii}.$$
	Inserting these into \eqref{Lv} and using \eqref{cond1}, it is immediate to check that there exists
	a positive constant $C_{8}$ depending only on $n, \text{diam}(\Omega), |\Omega|, \theta, \Theta$, $\max_{\overline{\Omega}}F[D^{2}u_{0}]$ and $\delta$, such that
	\begin{equation}\label{Lv estimate}
		|Lv|\leq C_{8}\sum_{i=1}^nF^{ii}.
	\end{equation}
	Combining \eqref{Ltildeh} with \eqref{Lv estimate} and letting $C_0$ be large enough, one yields
	$$\begin{aligned}
		L\Phi&=Lv+C_0L\tilde{h}+AL(|x-x_{0}^2|)\\
		&\leq C_{8}\sum_{i=1}^nF^{ii}-C_0\tilde{\theta}\sum_{i=1}^nF^{ii}-C_0\sum_{i=1}^nf_{p_i}\partial_i\tilde{h}+2A\lbrack\sum_{i=1}^nF^{ii}-f_{p_i}(x_i-x_{0i})\rbrack\\
		&\leq \left(C_{8}-\frac{C_{0}\tilde{\theta}}{2}+2A\right)\sum_{i=1}^nF^{ii}-2A\sum_{i=1}^n\left(f_{p_i}(x_i-x_{0i})\right)\\
		&\quad -  C_0\left(\frac{\tilde{\theta}}{2}\sum_{i=1}^nF^{ii}+f_{p_i}\partial_i\tilde{h}\right).
	\end{aligned}$$
	In order to make
	$$L\Phi\leq 0,$$
	where we have used the condition \eqref{lcond}.
	
	It is clear that  $\Phi\geq 0~ \text{on}~ \partial\Omega\times[0,T]$. Since $v$ is bounded, we can choose $A$ large enough (independent of $C_0$) such that in $\left(\Omega\cap\partial B_{\rho}(x_{0})\right)\times[0,T]$
	$$\Phi(x,t)=v(x,t)-v(x_0,t_0)+C_0\tilde{h}(x)+A|x-x_0|^2\geq v(x,t)-v(x_0,t_0)+A\rho^2\geq0.$$
	By the strict concavity of $\tilde{h}$ , we have
	$$\Delta(C_0\tilde{h}(x)+A|x-x_0|^2)\leq C(-C_0\tilde{\theta}+2A)\sum_{i=1}^nF^{ii}\quad in~ \mathscr{N}_\rho.$$
	Then by choosing $C_0\gg A$, we obtain
	$$\Delta(v(x,0)-v(x_0,t_0)+C_0\tilde{h}(x)+A|x-x_0|^2)\leq0.$$
	We apply the maximum principle to get
	$$\begin{aligned}
		&(v(x,0)-v(x_{0},t_{0})+C_{0}\tilde{h}(x)+A|x-x_{0}|^{2})|_{\mathscr{N}_{\rho}} \\
		\geq&\min_{(\partial\Omega\cap B_\rho(x_0))\cup(\Omega\cap\partial B_\rho(x_0))}(v(x,0)-v(x_0,t_0)+C_0\tilde{h}(x)+A|x-x_0|^2) \\
		\geq&\quad 0.
	\end{aligned}$$
	Using the maximum principle, we deduce that
	$$\Phi\geq0,\quad(x,t)\in\mathscr{N}_{\rho}\times[0,T].$$
	Combining it with $\Phi(x_0,t_0)=0$, we obtain $\langle\nabla\Phi,e_n\rangle|_{(x_0,t_0)} \geq 0$, which gives the
	desired key estimate \eqref{vn}.
\end{proof}

\begin{prop}\label{oblique estimate}
	Let $f\in\mathscr{A}_\delta$, satisfying \eqref{lcond} and 
	\begin{equation}\label{dfcond}
		|D_xf|\leq\frac{\theta \Lambda_1}{2}\cdot\frac{1}{\max_{\overline{\tilde{\Omega}}}|Dh|},
	\end{equation}
	If $u$ is a strictly convex solution to \eqref{EQP}--\eqref{IC}, then the uniformly obliqueness estimate
	\begin{equation}\label{obq eq}
		\langle\beta,\nu\rangle\geq\frac1{C_9}>0 \quad~on~\partial\Omega\times[0,T]
	\end{equation}
	holds for some universal constant $C_{9}$, which depends only on $n,  \text{diam}(\Omega), \text{diam}(\tilde{\Omega})$,\\$|\Omega|, |\tilde{\Omega}|$,
	$\theta,\tilde{\theta},\Theta,\tilde{\Theta}$,  $\max_{\overline{\Omega}}F[D^{2}u_{0}]$ and $\delta$.
\end{prop}

\begin{proof}

	Let $v$ and $(x_0,t_0)\in\partial\Omega\times[0,T]$ as the same in the proof of Lemma \ref{vnestimate}. 
	Combined the convexity of $\Omega$ with its smoothness, we extend $\nu$ smoothly to a tubular neighborhood
	of $\partial\Omega$ such that in the matrix sense
	\begin{equation}\label{normal vector}
		(\nu_{kl}):=(D_{k}\nu_{l})\leq-\frac{1}{C_{10}}\operatorname{diag}(1,\ldots,1,0),
	\end{equation}
	where $C_{10}$ is a positive constant depends only on $\Omega$. 
	
	At $(x_0, t_0)$, we have
	\begin{equation}\label{vr}
		0=v_{l}=h_{{p_{n}p_{k}}}u_{kl}+h_{{p_{k}}}\nu_{kl}+h_{{p_{k}}}u_{kl},\quad1\leq l\leq n-1.
	\end{equation}
	By Lemma \ref{vnestimate}, it is not hard to check that \eqref{vn} can be rewritten as
	\begin{equation}\label{vn1}
		h_{p_np_k}u_{kn}+h_{p_k}\nu_{kn}+h_{p_k}u_{kn}\geq-C_1.
	\end{equation}
	Multiplying \eqref{vn1} with $h_{p_n}$ and \eqref{vr} with $h_{p_r}$, respectively, and summing up together,
	we obtain
	\begin{equation}\label{hpkpl}
		h_{{p_{k}}}h_{{p_{l}}}u_{kl}\geq-C_{1}h_{{p_{n}}}-h_{{p_{k}}}h_{{p_{l}}}\nu_{kl}-h_{{p_{k}}}h_{{p_{n}p_{l}}}u_{kl}.
	\end{equation}
	Using \eqref{normal vector}, and
	$$1\leq r\leq n-1,\quad h_{{p_{k}}}u_{kr}=\frac{\partial h(Du)}{\partial x_{r}}=0,\quad h_{{p_{k}}}u_{kn}=\frac{\partial h(Du)}{\partial x_{n}}\geq0,\quad-h_{{p_{n}p_{n}}}\geq0,$$
	we have
	$$h_{{p_{k}}}h_{{p_{l}}}u_{kl}\geq-C_{1}h_{{p_{n}}}+\frac{1}{C_{10}}|Dh|^{2}-\frac{1}{C_{10}}h_{{p_{n}}}^{2}\geq-C_{11}h_{{p_{n}}}+\frac{1}{C_{11}}-\frac{1}{C_{11}}h_{{p_{n}}}^{2},$$
	where we use $|Dh|^2-h_{p_n}^2=\sum_{k=1}^{n-1}h_{p_k}^2$ and $C_{11}=\max\{C_1,C_{10}\}$. For the last term
	of the above inequality, we distinguish two cases at $(x_0, t_0)$.
	
	Case (i). If
	$$-C_{11}h_{p_n}+\frac1{C_{11}}-\frac1{C_{11}}h_{p_n}^2\leq\frac1{2C_{11}},$$
	then
	$$h_{p_k}(Du)\nu_k=h_{p_n}\geq\sqrt{\frac12+\frac{C_{11}^4}4}-\frac{C_{11}^2}2.$$
	It shows that there is a uniform positive lower bound for the quantity\\ $\min_{\partial\Omega\times[0,T]}h_{{p_{k}}}(Du)\nu_{k}$.
	
	Case (ii). If
	$$-C_{11}h_{p_n}+\frac1{C_{11}}-\frac1{C_{11}}h_{p_n}^2>\frac1{2C_{11}},$$
	then we obtain a positive lower bound of $h_{p_k}h_{p_l}u_{kl}$.
	
	Next, we consider the Legender transformation $\tilde{u}$. By Definition \ref{defh}, it follows that the unit inward normal vector of $\partial\Omega$ can be expressed by $\nu=D\tilde{h}.$ For the same reason, $\tilde{\nu}=Dh$, where $\tilde{\nu}=(\tilde{\nu}_{1},\tilde{\nu}_{2},\ldots,\tilde{\nu}_{n})$ is the unit inward normal vector of $\partial\tilde{\Omega}$.
	Let $\tilde{\beta}=(\tilde{\beta}^1,\ldots,\tilde{\beta}^n)$ with $\tilde{\beta}^k:=\tilde{h}_{\tilde{p}_k}(D\tilde{u})$. We note that one can also define
	$$\tilde{v}=\langle\tilde{\beta},\tilde{\nu}\rangle+\tilde{h}(D\tilde{u}),$$
	in which
	$$\langle\tilde{\beta},\tilde{\nu}\rangle=\langle\beta,\nu\rangle.$$
	Denote $ \tilde{x}_0 = Du(x_0)$. Then, we have
	$$\tilde{v}(\tilde{x}_0,t_0)=v(x_0,t_0)=\min_{\partial\tilde{\Omega}\times[0,T]}\tilde{v}.$$
	Similar to the proof of Lemma \ref{vnestimate}, under the assumption of \eqref{dfcond}, we can get
	\begin{equation}\label{tildevn}
		\tilde{v}_n(\tilde{x}_0,t_0)\geq-C_{12},
	\end{equation}
	where $C_{12}$ is a constant depending only on $n,  \text{diam}(\tilde{\Omega}),|\tilde{\Omega}|, \theta,\tilde{\theta}, \tilde{\Theta}$,  $\max_{\overline{\Omega}}F[D^{2}u_{0}] $ and $\delta$. By the same arguement, we obtain the positive lower bounds of $\tilde{h}_{p_k}\tilde{h}_{p_l}\tilde{u}_{kl}$, or
	$$h_{p_{k}}(Du)\nu_{k}=\tilde{h}_{p_{k}}(D\tilde{u})\tilde{\nu}_{k}=\tilde{h}_{p_{n}}\geq\sqrt{\frac{1}{2}+\frac{C_{12}^{4}}{4}}-\frac{C_{12}^{2}}{2}.$$
	It follows that
	$$\tilde{h}_{p_k}\tilde{h}_{p_l}\tilde{u}_{kl}=\nu_i\nu_ju^{ij},$$
	where $(u^{ij})=(u_{ij})^{-1} $, then by the positive lower bounds of $h_{p_k}h_{p_l}u_{kl}$ and $\tilde{h}_{p_k}\tilde{h}_{p_l}\tilde{u}_{kl}$, the desired result	follows from
	\begin{equation}\label{obq}
		\langle\beta,\nu\rangle=\sqrt{u^{ij}\nu_i\nu_jh_{pk}h_{pl}u_{kl}}\geq C_{13},
	\end{equation}
	which is proved in \cite{Urbas1997OnTS}. Finally, set $C_{9}=C_{13}$, we get the uniformly obliqueness estimate.
\end{proof}
\section{ $C^2$ estimate}\label{sec4}
We now proceed to carry out the global $C^2$ estimate. The $C^2$ a priori bound is accomplished by making the second derivative
estimates on the boundary for the solutions of fully nonlinear parabolic equations. 

The strategy is to reduce the
$C^2$ global estimate of $u$ and $\tilde{u}$ to the boundary. 
\begin{lem}\label{C2estimate}
	Let $f\in \mathscr{A}_\delta$,  satisfying \eqref{lcond} and
	\begin{equation}\label{Dxp}
		\max_{\overline{\Omega}\times\overline{\tilde{\Omega}}}|D_{xp}f|\leq\frac{\Lambda_1\tilde{\theta}}{8\max_{\overline{\Omega}}|\tilde{h}|}.
	\end{equation}
	If $u$ is a strictly convex solution of \eqref{evoproblem}, then there exists a positive constant $C_{14}$ depending only on $n, \text{diam}(\Omega),  |\Omega|,\tilde{\theta},\max_{\overline{\Omega}}F[D^{2}u_{0}]$ and $\delta$, such that
	\begin{equation}\sup_{{\Omega_{T}}}|D^{2}u|\leq2 \max_{\partial\Omega\times[0,T]}|D^{2}u|+2\max_{{\overline{\Omega}}}|D^{2}u_{0}|+C_{14}\max_{ \overline{\Omega}\times\overline{\tilde{\Omega}}}|D_{xx}f|.
	\end{equation}
\end{lem}
\begin{proof}
	Let
	$$L:=F^{ij}\partial_{ij}-f_{p_i}\partial_i-\partial_t.$$
	For any unit vector $\xi$ , differentiating the equation in \eqref{evoproblem} twice in direction $\xi$ gives
	$$Lu_{\xi\xi}+F^{ij,rs}u_{ij\xi}u_{rs\xi}=f_{\xi\xi}+2f_{\xi p_i}u_{i\xi}+f_{p_ip_j}u_{i\xi}u_{j\xi}\quad in~\Omega_{T}.$$
	Then by the concavity of $F$ on $\Gamma_{n}^{+}$ and the convexity of $f$ in $p$, we have
	\begin{align*}
		Lu_{\xi\xi}&\geq f_{\xi\xi}+2f_{\xi p_i}u_{i\xi}\\
		&\geq f_{\xi\xi}-2\max_{\overline{\Omega}\times\overline{\tilde{\Omega}}}|D_{xp}f|\sup_{{\Omega_{T}}}|D^2u|.
	\end{align*}
	Let
	$$v=\sup_{\partial_p \Omega_{T}}u_{\xi\xi}+\frac{2}{\Lambda_1\tilde{\theta}}\left(\max_{ \overline{\Omega}\times\overline{\tilde{\Omega}}}|D_{xx}f|+2\max_{\overline{\Omega}\times\overline{\tilde{\Omega}}}|D_{xp}f|\sup_{{\Omega_{T}}}|D^2u|\right)\tilde{h}.$$
	By direct calculation and \eqref{cond1}, \eqref{lcond}, \eqref{Ltildeh}, we obtain
	\begin{align*}
		Lv&\leq\frac{2}{\tilde{\theta}\Lambda_1}\left(\max_{ \overline{\Omega}\times\overline{\tilde{\Omega}}}|D_{xx}f|+2\max_{\overline{\Omega}\times\overline{\tilde{\Omega}}}|D_{xp}f|\sup_{{\Omega_{T}}}|D^2u|\right)\left(-\tilde{\theta}\sum_{i=1}^nF^{ii}-f_{p_i}\partial_i\tilde{h}\right)\\
		&\leq -\max_{ \overline{\Omega}\times\overline{\tilde{\Omega}}}|D_{xx}f|-2\max_{\overline{\Omega}\times\overline{\tilde{\Omega}}}|D_{xp}f|\sup_{{\Omega_{T}}}|D^2u|\quad in~\Omega_{T}\\
	\end{align*}
	and thus
	$$L(v-u_{\xi\xi})\leq0\quad in~\Omega_{T}.$$
	It is obvious that $v-u_{\xi\xi}\geq 0$ on $\partial_p \Omega_{T}$. Then, by the maximum principle, we obtain
	\begin{align*}
		\sup_{\Omega_T}u_{\xi\xi} & \leq\sup_{{\Omega_{T}}}v\leq\sup_{{\partial_p \Omega_{T}}}u_{\xi\xi}+\frac{2}{\Lambda_1\tilde{\theta}}\max_{\overline{\Omega}}|\tilde{h}|\left(\max_{ \overline{\Omega}\times\overline{\tilde{\Omega}}}|D_{xx}f|+2\max_{\overline{\Omega}\times\overline{\tilde{\Omega}}}|D_{xp}f|\sup_{{\Omega_{T}}}|D^2u|\right) \\
		&\leq\max_{\partial\Omega\times[0,T]}|D^2u|+\max_{\overline{\Omega}}|D^2u_0|+\frac{2}{\Lambda_1\tilde{\theta}}\max_{\overline{\Omega}}|\tilde{h}|\max_{ \overline{\Omega}\times\overline{\tilde{\Omega}}}|D_{xx}f|\\
		&\quad +\frac{4}{\Lambda_1\tilde{\theta}}\max_{\overline{\Omega}}|\tilde{h}|\max_{\overline{\Omega}\times\overline{\tilde{\Omega}}}|D_{xp}f|\sup_{{\Omega_{T}}}|D^2u|.
	\end{align*}
	by the condition of  \eqref{Dxp},
	we complete the proof of \eqref{C2estimate}.
\end{proof}

Next, we estimate the second order derivative on the boundary. By differentiating
the boundary condition $h(Du) = 0$ in any tangential direction $\varsigma$, we have
\begin{equation}\label{ubetasigma}
	u_{\beta \varsigma}=h_{p_k}(Du)u_{k\varsigma}=0.
\end{equation}
The second order derivative of $u$ on the boundary is controlled by $u_{\beta \varsigma}, u_{\beta\beta} ~and~ u_{\varsigma\varsigma}$.
In the following, we give the arguments as in \cite{Urbas1997OnTS}, one can see there for more details.

At $x \in \partial\Omega$, any unit vector $\xi$ can be written in terms of a tangential component
$\varsigma(\xi)$ and a component in the direction $\beta$ by
$$\xi=\varsigma(\xi)+\frac{\langle\nu,\xi\rangle}{\langle\beta,\nu\rangle}\beta,$$
indeed, $\varsigma(\xi)$ can be expressed by
$$\varsigma(\xi)=\xi-\langle\nu,\xi\rangle\nu-\frac{\langle\nu,\xi\rangle}{\langle\beta,\nu\rangle}\beta^T,$$
and
$$\beta^{T}=\beta-\langle\beta,\nu\rangle\nu.$$
By the uniformly obliqueness estimate \eqref{obq eq}, we have
\begin{equation}\label{varsigma}
	|\varsigma(\xi)|^{2}\leq C_{15}.
\end{equation}
Denote $\varsigma:=\frac{\varsigma(\xi)}{|\varsigma(\xi)|}$ then by \eqref{ubetasigma}, \eqref{varsigma} and \eqref{obq eq}, we obtain
\begin{equation}\label{eq3.21}
	\begin{aligned}
		u_{\xi\xi} & =|\varsigma(\xi)|^2u_{\varsigma\varsigma}+2|\zeta(\xi)|\frac{\langle\nu,\xi\rangle}{\langle\beta,\nu\rangle}u_{\beta\varsigma}+\frac{\langle\nu,\xi\rangle^2}{\langle\beta,\nu\rangle^2}u_{\beta\beta} \\
		&=|\varsigma(\xi)|^2u_{\varsigma\varsigma}+\frac{\langle\nu,\xi\rangle^2}{\langle\beta,\nu\rangle^2}u_{\beta\beta} \\
		&\leq C_{16}(u_{\varsigma\varsigma}+u_{\beta\beta}),
	\end{aligned}
\end{equation}
where $C_{17}$ depends only on $n, \text{diam}(\Omega), \text{diam}(\tilde{\Omega}), |\Omega|, |\tilde{\Omega}|,\theta,\tilde{\theta},\Theta,\tilde{\Theta}$, $\max_{\overline{\Omega}}F[D^{2}u_{0}]$ and $\delta$. Therefore,
we only need to estimate $u_\mathrm{\beta\beta}$ and $u_\mathrm{\varsigma\varsigma}$, respectively.

Similar to Proposition 2.6 in \cite{Brendle2008ABV}, by making use of \eqref{Lv estimate}, we can obtain
\begin{lem}\label{Lphistimate}
	If $u$
	is a strictly convex solution to \eqref{evoproblem} and $|D_xf|,|D_pf|$ satisfy \eqref{dfcond}, \eqref{lcond}, respectively. Fix a smooth function $\eta: \Omega \times \tilde{\Omega } \to \mathbb{R}$  defined by $\varphi(x,t)=\eta(x,
	Du(x,t))$, it follows that
	$$|L\varphi|\leq C_{17}\sum_{i=1}^nF^{ii},$$
	where $C_{17}$ depends only on $n,  \text{diam}(\Omega), |\Omega|, \theta, \Theta, \max_{\overline{\Omega}}F[D^{2}u_{0}],\delta$ and $ \|\eta\|_{C^2(\overline{\Omega}\times\overline{\tilde{\Omega}})} $.
\end{lem}

Further, we have
\begin{lem}\label{ubetabeta}
	Let $f\in \mathscr{A}_{\delta}$ and $|D_pf|$ satisfies \eqref{lcond}. If $u$ is a strictly convex solution of \eqref{evoproblem}, then there exists a positive constant $C_{18}$ depending only on $n, \text{diam}(\Omega),  |\Omega|, \theta,\Theta,\tilde{\theta}$,  $\max_{\overline{\Omega}}F[D^{2}u_{0}]$ and $\delta$, such that
	\begin{equation}\label{ubetabetaesimate}
		\max_{\partial\Omega \times [0, T]}u_{\beta\beta}\leq C_{18}.
	\end{equation}
\end{lem}
\begin{proof}
	Let $x_0 \in\partial\Omega$, $t_0 \in [0, T]$ satisfy $u_{\beta\beta}(x_0,t_0) = \max_{\partial\Omega \times [0, T]}u_{\beta\beta}$ and we also denote $\mathscr{N}_{\rho}$ and $L$ as in Lemma \ref{vnestimate}. 
	
	Consider the barrier function
	$$\Psi:=-h(Du)+C_0\tilde{h}+A|x-x_0|^2.$$
	It follows the same arguement in the proof of Lemma \ref{vnestimate} that we can choose $A$ large enough and  $C_0\gg A$ such that
	$$\Psi\geq0,\quad on~\partial\mathscr{N}_{\rho}\times[0,T].$$

	Noting that  $h$ is a smooth function depending on $Du$, by Lemma  \ref{Lphistimate} there exists $C$ depending only on $n, \text{diam}(\Omega), |\Omega|,  \theta, \Theta, \max_{\overline{\Omega}}F[D^2u_{0}] $ and $\delta$, such that 
	$$|Lh|\leq C\sum_{i=1}^{n}F^{ii}.$$
	Letting $C_0\gg A$  and using \eqref{lcond}, \eqref{Ltildeh}, one yields
	$$\begin{aligned}
		L\Psi&=-Lh+C_0L\tilde{h}+AL(|x-x_{0}|^2)\\
		&\leq \left(C+2A-\frac{C_0\tilde{\theta}}{2}\right)\sum_{i=1}^{n}F^{ii}++2A\sum_{i=1}^{n}f_{p_i}(x_i-x_{0i})\\
		&\quad -C_0\left(\frac{\tilde{\theta}}{2}\sum_{i=1}^{n}F^{ii}+f_{p_i}\partial_i\tilde{h}\right),\\
		&\leq 0,\quad \quad  \quad \quad \quad in~\mathscr{N}_{\rho}\times[0,T].
	\end{aligned}$$

	
	Using the maximum principle, we deduce that
	$$\Psi\geq0,\quad in~\mathscr{N}_{\rho}\times[0,T].$$
	Combining it with $\Psi(x_0,t_0)=0$, we obtain $\Psi_\beta(x_0,t_0)\geq0$,  which implies
	$$\frac{\partial h}{\partial\beta}(Du(x_0,t_0))\leq C_{18}.$$
	On the other hand, we see that at $(x_0, t_0)$,
	$$\frac{\partial h}{\partial\beta}=\langle Dh(Du),\beta\rangle=\frac{\partial h}{\partial p_k}u_{kl}\beta^l=\beta^ku_{kl}\beta^l=u_{\beta\beta}.$$
	Therefore,
	$$u_{\beta\beta}=\frac{\partial h}{\partial\beta}\leq C_{18},$$
	whence the result follows.
\end{proof}

Next, we estimate the double tangential derivative.
\begin{lem}\label{utangent}
	Let $f\in \mathscr{A}_{\delta}$, $|D_xf|,|D_pf|,|D_{xp}f|$ satisfy \eqref{dfcond}, \eqref{lcond}, \eqref{Dxp}, respectively. If $u$ is a strictly convex solution of \eqref{evoproblem}, then there exists a positive constant $C_{19}$ depending only on $n, \text{diam}(\Omega),$ $ \text{diam}(\tilde{\Omega}), |\Omega|,|\tilde{\Omega}|,\theta,\tilde{\theta},\Theta,\tilde{\Theta},\max_{\overline{\Omega}}F[D^{2}u_{0}]$ and $\delta$, such that
	\begin{equation}\label{uvarsigma}
		\max_{\partial\Omega\times[0,T]}u_{\varsigma\varsigma}\leq C_{19}.
	\end{equation}
\end{lem}
\begin{proof}
	Without loss of generality, we assume that $x_0\in\partial\Omega,t_0\in(0,T]$, and $e_n,e_1$ are the unit inward normal vector and unit tangential vector at $(x_0,t_0)$, such that
	$$\max_{\partial\Omega\times[0,T]}u_{\varsigma\varsigma}=u_{11}(x_0,t_0)=:\mathcal{M}.$$
	For any $x\in \Omega$, by uniformly oblique estimate \eqref{obq eq}, we have
	\begin{equation}\label{eq3.25}
		\begin{aligned}
			u_{\xi\xi} & =|\varsigma(\xi)|^2u_{\varsigma\varsigma}+\frac{\langle\nu,\xi\rangle^2}{\langle\beta,\nu\rangle^2}u_{\beta\beta} \\
			&\leq \left(1+C_{20}\langle\nu,\xi\rangle^{2}-2\langle\nu,\xi\rangle\frac{\langle\beta^{T},\xi\rangle}{\langle\beta,\nu\rangle}\right)\mathcal{M}+\frac{\langle\nu,\xi\rangle^{2}}{\langle\beta,\nu\rangle^{2}}u_{\beta\beta}.,
		\end{aligned}
	\end{equation}
	Without loss of generality, we assume that $\mathcal{M} \geq 1$. Then by \eqref{obq eq} and \eqref{ubetabetaesimate}, one yields
	\begin{equation}\label{eq3.26}
		\frac{u_{\xi\xi}}{\mathcal{M}}+2\langle\nu,\xi\rangle\frac{\langle\beta^{T},\xi\rangle}{\langle\beta,\nu\rangle}\leq1+C_{21}\langle\nu,\xi\rangle^{2}.
	\end{equation}
	Let $ \xi = e_1$, then
	\begin{equation}\label{eq3.27}
		\frac{u_{11}}{\mathcal{M}}+2\langle\nu,e_1\rangle\frac{\langle\beta^{T},e_1\rangle}{\langle\beta,\nu\rangle}\leq1+C_{21}\langle\nu,e_1\rangle^{2}.
	\end{equation}
	As in the proof of Proposition 2.14 in \cite{Brendle2008ABV}, let $\chi:\mathbb{R}\to\mathbb{R}$ be a smooth cutoff function satisfying $\chi(s)=s$ for $s\geq\frac1{C_1}$ and $\chi(s)\geq\frac1{2C_1}$ for all $s\in\mathbb{R}$.
	We see that the function
	\begin{equation}\label{w}
		w:=A|x-x_{0}|^{2}-\frac{u_{11}}{\mathcal{M}}-2\langle\nu,e_{1}\rangle\frac{\langle\beta^{T},e_{1}\rangle}{\chi(\langle\beta,\nu\rangle)}+C_{21}\langle\nu,e_{1}\rangle^{2}+1
	\end{equation}
	satisfies
	$$w|_{\partial\Omega\times[0,T]}\geq0 \quad and \quad  w(x_0,t_0)=0.$$
	Then, it follows the same arguement in Lemma \ref{vnestimate} that we can choose the constant $A$ large enough such that
	$$w|_{(\partial B_\rho(x_0)\cap\Omega)\times[0,T]}\geq0.$$
	
	Consider
	$$-2\langle\nu,e_{1}\rangle\frac{\langle\beta^{T},e_{1}\rangle}{\chi(\langle\beta,\nu\rangle)}+C_{21}\langle\nu,e_{1}\rangle^{2}+1$$
	as a known smooth function depending on $x$ and $Du$. Then by Lemma \ref{Lphistimate}, we obtain
	$$\left|L\left(-2\langle\nu,e_1\rangle\frac{\langle\beta^T,e_1\rangle}{\chi(\langle\beta,\nu\rangle)}+C_{21}\langle\nu,e_1\rangle^2+1\right)\right|\leq C_{22}\sum_{i=1}^nF^{ii}.$$
	By making of the concavity of $F$ and the convexity of $f$ in $p$,  we have
	\begin{align*}
		Lu_{11}\geq f_{11}-2\max_{\overline{\Omega}\times\overline{\tilde{\Omega}}}|D_{xp}f|\sup_{\Omega_{T}}|D^2u|.
	\end{align*}
	Combining the Lemmas \ref{C2estimate}, \ref{ubetabeta} and \eqref{eq3.21}, we find
	\begin{align*}
		Lw&\leq2A\sum_{i=1}^nF^{ii}-2Af_{p_i}(x_i-x_{0i})-\frac{1}{\mathcal{M}}\left(f_{11}-2\max_{\overline{\Omega}\times\overline{\tilde{\Omega}}}|D_{xp}^2f|\sup_{\Omega_{T}}|D^2u|\right)+C_{22}\sum_{i=1}^nF^{ii}\\
		& \leq C_{23}\sum_{i=1}^nF^{ii}.
	\end{align*}
	As in the proof of Lemma \ref{ubetabeta}, we consider the function
	$$\Upsilon:=w+C_0\tilde{h}.$$
	By a simple calculation, we can get
	$$\Upsilon\geq0 \quad on~(\partial\mathscr{N}_{\rho}\times[0,T])\cup(\mathscr{N}_{\rho}\times\{t=0\}),$$
	and it follows by \eqref{lcond}, we have
	\begin{align*}
		L\Upsilon&=Lw+C_0L\tilde{h}\\
		&\leq \left(C_{23}-\frac{C_0\tilde{\theta}}{2} \right)\sum_{i=1}^nF^{ii}.
	\end{align*}
	Then, letting $C_0$ be large enough, one yields 
	$$L\Upsilon\leq 0  \quad in \mathscr{N}_{\rho}\times[0,T].$$

	Combining the maximum principle with $\Upsilon(x_0,t_0)=0$, we obtain $\Upsilon_\beta(x_0,t_0)\geq0$,  which implies
	\begin{equation}\label{u11beta}
		u_{11\beta}(x_0,t_0)\leq C_{24}\mathcal{M}.
	\end{equation}
	On the other hand, differentiating $h(Du)$ twice in the direction $e_1$ at $(x_0, t_0)$, we have
	$$h_{{p_{k}}}u_{k11}+h_{{p_{k}p_{l}}}u_{k1}u_{l1}=0.$$
	The concavity of $h$ yields that
	$$h_{{p_{k}}}u_{k11}=-h_{{p_{k}p_{l}}}u_{k1}u_{l1}\geq \theta|Du_1|^2\geq\theta \mathcal{M}^{2}.$$
	Combining it with $ h_{ p_k}u_{k11} = u_{11\beta}$, and using \eqref{u11beta}, we obtain
	$$\theta \mathcal{M}^2\leq C_{24}\mathcal{M}.$$
	Then, we get the upper bound of $\mathcal{M} = u_{11}(x_0, t_0)$ and thus the desired result
	follows.
\end{proof}

By Lemma \ref{ubetabeta}, \ref{utangent} and \eqref{eq3.21}, we obtain the $C^2$ a priori estimate on
the boundary.
\begin{lem}\label{C2bound}
	Let $f\in \mathscr{A}_{\delta}$  and $|D_xf|,|D_pf|,|D_{xp}f|$ satisfy \eqref{dfcond}, \eqref{lcond}, \eqref{Dxp}, respectively. If $u$ is a strictly convex solution of \eqref{evoproblem}, then there exists a positive constant $C_{25}$ depending only on $n, \text{diam}(\Omega), \text{diam}(\tilde{\Omega}), |\Omega|, |\tilde{\Omega}|,\theta,\tilde{\theta},\Theta,\tilde{\Theta}, \max_{\overline{\Omega}}F[D^{2}u_{0}]$ and $\delta$, such that
	\begin{equation}\label{D2upartialomega}
		\max_{\partial \Omega\times[0,T]}|D^2u|\leq C_{25}.
	\end{equation}
\end{lem}

In terms of Lemmas \ref{C2estimate} and \ref{C2bound}, we readily conclude:
\begin{lem}\label{C2globalbound}
	Let $f\in \mathscr{A}_{\delta}$  and $|D_xf|,|D_pf|,|D_{xp}f|$ satisfy \eqref{dfcond}, \eqref{lcond}, \eqref{Dxp}, respectively. If $u$ is a strictly convex solution of \eqref{evoproblem}, then there exists a positive constant $C_{26}$ depending only on $n, \text{diam}(\Omega), \text{diam}(\tilde{\Omega}), |\Omega|, |\tilde{\Omega}|,\theta,\tilde{\theta},\Theta,\tilde{\Theta},\|u_0\|_{C^2(\overline{\Omega})}$,\\
	$\max_{\overline{\Omega}}F[D^{2}u_{0}]$ and $\delta$, such that
	\begin{equation}\label{D2upartialbaromega}
		\max_{\overline{\Omega_{T}}}|D^2u|\leq C_{26}.
	\end{equation}
\end{lem}

In the following, we describe the positive lower bound of $D^2u$. 
By \eqref{concavemu}, we consider the dual problem \eqref{tilde u eq1}, using the same arguement in the proof of Lemma \ref{C2globalbound} with operator
$$\tilde{L}:=\tilde{F}^{ij}\partial_{ij}+f_{\tilde{p}_i}\partial_i-\partial_t,$$
we can obtain

\begin{lem}\label{tildeC2globalbound}
	Let $f\in \mathscr{A}_{\delta}$, $|D_xf|, |D_pf|$ satisfy \eqref{dfcond}, \eqref{lcond} and  
	\begin{equation}\label{Dxp1}
		\max_{ \overline{\Omega}\times\overline{\tilde{\Omega}}}|D_{xp}f|\leq\frac{\Lambda_1\theta}{8\max_{\overline{\tilde{\Omega}}}|h|}.
	\end{equation}
	If $\tilde{u}$ is a strictly convex solution of \eqref{tilde u eq1}, then there exists a positive constant $C_{38}$ depending only on $n, \text{diam}(\Omega),  \text{diam}(\tilde{\Omega}), |\Omega|, |\tilde{\Omega}|,\theta,\tilde{\theta},\Theta,\tilde{\Theta},\|u_0\|_{C^2(\overline{\Omega})},\max_{\overline{\Omega}}F[D^{2}u_{0}]$ and $\delta$, such that
	\begin{equation}\label{tildeD2upartialbaromega}
		\max_{\overline{\tilde{\Omega}_{T}}}|D^2\tilde{u}|\leq C_{27}.
	\end{equation}
\end{lem}
Finally, combined with Lemma \ref{C2globalbound} and Lemma \ref{tildeC2globalbound}, we conclude:
\begin{lem}\label{globalC2}
	Let $f\in \mathscr{A}_{\delta}$ and $|D_xf|, |D_pf|, |D_{xp}f|$ satisfy \eqref{dfcond}, \eqref{lcond}, \eqref{Dxp} and \eqref{Dxp1}. If $u$ is a strictly convex solution of \eqref{evoproblem}, then there exists a positive constant $C_{28}$ depending only on $n, \text{diam}(\Omega),  \text{diam}(\tilde{\Omega}), |\Omega|, |\tilde{\Omega}|,\theta,\tilde{\theta},\Theta,\tilde{\Theta},\|u_0\|_{C^2(\overline{\Omega})}$,\\
	$\max_{\overline{\Omega}}F[D^{2}u_{0}]$ and $\delta$, such that
	\begin{equation}\label{globalC2estimate}
		\frac1{C_{28}}I_n\leq D^2u(x,t)\leq C_{28}I_n,\quad(x,t)\in\overline{\Omega_{T}},
	\end{equation}
	where $I_n$ is the $n\times n$ identity matrix.
\end{lem}
Finally, we summarize the proof of the \( C^2 \) estimate. The convexity condition of \( f \) in \( Du \) and the smallness of \( |D_{xp}^2 f| \) are used to reduce the positive upper bound of \( D^2 u \) to the boundedness on the boundary and to estimate the pure tangential second derivatives of the solution on the boundary (see Lemmas \ref{C2estimate} and \ref{utangent}). Combining this with the Legendre transformation of \( u \), the concavity condition of \( f \) in \( x \), and the smallness of \( |D_{xp}^2 f| \), we further reduce the positive lower bound of \( D^2 u \) to the boundedness on the boundary and estimate the pure tangential second derivatives of the dual solution on the boundary.

\section{Longtime existence and convergence}\label{section5}

Now, we give
a proof of Theorem \ref{main-thm}. This a standard result by our $C^2$ estimates and uniformly
oblique estimates, but for convenience we include here a proof.

Part 1: The long time existence.

Using Theorem 14.22 in Lieberman \cite{lieberman} and Proposition \ref{oblique estimate}, we can show that the solutions of uniformly oblique derivative problem \eqref{evoproblem} have global $C^{2+\alpha_0,1+\frac{\alpha_0}2}$ estimates, for some $0<\alpha_0<\alpha$.
Now, let $u_{0}$ be a $C^{2+\alpha}$ strictly convex function as in the conditions of Theorem \ref{main-thm}.  Combining Proposition \ref{shorttime} with Lemma \ref{globalC2} and using Theorem 14.23 in \cite{lieberman}, hence a standard argument using the Arzelà--Ascoli theorem (see \cite{Wang2023} for more details) gives the existence of a $C^{2+\alpha_0,1+\frac{\alpha_0}2}$ solutionu of \eqref{evoproblem} for all times $t>0$ and $0<\alpha_0<\alpha$.

Part 2: The convergence.

The standard Schauder estimates \cite{lieberman,Ladyzhenskaya} imply uniform bounds of the form
\begin{align*}
	&\|D^3u(\cdot,t)\|_{C(\overline{\Omega})}+\|D^4u(\cdot,t)\|_{C(\overline{\Omega})}\\
	&+\sup_{x_1,x_2\in \overline{\Omega},(x_1,t_1)\neq (x_2,t_2)}\frac{|D^4u(x_1,t_1)-D^4u(x_2,t_2)|}{|x_1-x_2|^{\alpha_0}+|t_1-t_2|^{\frac{\alpha_0}2}}\leq\tilde{C}_1,
\end{align*}
where $\tilde{C}_{1}$ are constants depending on $n, \text{diam}(\Omega),  \text{diam}(\tilde{\Omega}), |\Omega|, |\tilde{\Omega}|,\theta,\tilde{\theta},\Theta,\tilde{\Theta}$,\\
$ \|u_0\|_{C^{2+\alpha}(\overline{\Omega})},$
$ \max_{\overline{\Omega}}F[D^{2}u_{0}],$ and $\delta$.

Now following Section 6.2 in \cite{Schnürer2002}  and Theorem 1.1 in \cite{Huang2021}, we can obtain a translating solution of the same regularity as $u$, i.e. a function $u^\infty(x,t)=\tilde{u}^\infty+C_\infty\cdot t$ for some constant $c_\infty$ that satisfes the problem \eqref{EQP}-\eqref{BC}, such that $\|u-u^\infty\|_{C^{4,\gamma}(\overline{\Omega})}\to0$ as $t\to\infty$ for any $0<\gamma<\alpha_0$. Thus 
$$c_\infty=F[D^2\tilde{u}^\infty]-f(x,D\tilde{u}^\infty)\quad x\in \Omega, \quad and \quad  D\tilde{u}^\infty(\Omega)=\tilde{\Omega}.$$
Combining the Theorem 1.1 in \cite{lieberman1986} and bootstrap arguments, we finish the proof of Theorem \ref{main-thm}.
\begin{cor}
	By  Remark \ref{Adelta}, we see that Theorem \ref{prethm}
	is a direct consequence of Theorem \ref{main-thm}.
\end{cor}

\end{document}